\documentclass[runningheads]{llncs}

\usepackage{amsmath}
\usepackage{amsfonts}

\usepackage{amsthm}
\usepackage{dsfont}

\usepackage{booktabs}
\usepackage{cite}

\usepackage{algorithm}
\usepackage{algpseudocode}

\usepackage{graphicx}
\usepackage{caption}
\usepackage{subcaption}

\usepackage{multirow}

\graphicspath{{./images/}}

\usepackage{hyperref}

\theoremstyle{definition}
\newtheorem{assumption}{Assumption}

\newcommand{\EE}[2][]{\mathbb{E}_{#1}\left[#2\right]} 
\newcommand{\PP}[2][]{\mathbb{P}_{#1}\left[#2\right]} 

\newcommand{\II}[1]{\mathbb{I}\left[#1\right]}        

\newcommand{\B}{\mathcal{B}}
\newcommand{\C}{\mathcal{C}}
\newcommand{\I}{\mathcal{I}}
\newcommand{\J}{\mathcal{J}}
\renewcommand{\L}{\mathcal{L}}
\newcommand{\N}{\mathcal{N}}
\renewcommand{\S}{\mathcal{S}}
\newcommand{\T}{\mathcal{T}}
\newcommand{\U}{\mathcal{U}}
\newcommand{\X}{\mathcal{X}}

\newcommand{\eps}{\varepsilon}

\newcommand{\IR}{\mathds{R}}

\DeclareMathOperator*{\argmax}{\mathrm{argmax}}

\DeclareMathOperator{\diam}{\mathrm{diam}}

\providecommand{\abs}[1]{\left\lvert#1\right\rvert} 

\providecommand{\set}[1]{\left\lbrace#1\right\rbrace}

\newcommand{\defin}{\stackrel{\mathrm{def}}{=}}

\begin{document}

\title{Limited depth bandit-based strategy for Monte Carlo planning in continuous action spaces}
\titlerunning{Limited depth bandits for MC planning}


\author{Ricardo Quinteiro\inst{1} \and
Francisco S.\ Melo\inst{1,2} \and
Pedro A.\ Santos\inst{1,2}}

\authorrunning{R.\ Quinteiro et al.}

\institute{Instituto Superior T\'{e}cnico, University of Lisbon, Portugal\\
\and
INESC-ID, Lisbon, Portugal}

\maketitle

\begin{abstract}

This paper addresses the problem of optimal control using search trees. We start by considering multi-armed bandit problems with continuous action spaces and propose LD-HOO, a limited depth variant of the hierarchical optimistic optimization (HOO) algorithm. We provide a regret analysis for LD-HOO and show that, asymptotically, our algorithm exhibits the same cumulative regret as the original HOO while being faster and more memory efficient. We then propose a Monte Carlo tree search algorithm based on LD-HOO for optimal control problems and illustrate the resulting approach's application in several optimal control problems. 

\end{abstract}


\section{Introduction}%
\label{Sec:Intro}

Recent years have witnessed remarkable successes of artificial intelligence, brought about by the effective combination of powerful machine learning algorithms based on neural networks and efficient planning approaches based on Monte Carlo tree search. Perhaps the most celebrated is AlphaGo and its successors \cite{silver16nature,silver17nature,silver18science,schrittwieser20nature}, which could attain super-human proficiency in several complex games such as chess, Go, or Shogi. Even before AlphaGo, Monte Carlo tree search approaches were already widely used in complex planning problems due to their efficiency, mild requirements, and robust theoretical guarantees \cite{browne12tcig}	.
 
Monte Carlo tree search (MCTS) is a family of online planning methods initially designed to solve sequential decision problems such as Markov decision problems \cite{kocsis06ecml,kearns99ijcai}. At each step $t$, the decision-maker builds a search tree with its current state, $s_t$, as the root node. Its children correspond to the actions available to the agent at $s_t$, and the algorithm proceeds by estimating a value for each such action by using {\em rollouts} and {\em efficient exploration}. Rollouts provide empirical estimates for the nodes' values in the tree, while efficient exploration directs tree expansion towards the more promising regions of the state space (see Fig.~\ref{Fig:MCTS} for an illustration).

\begin{figure}[!tb]
\centering
\includegraphics{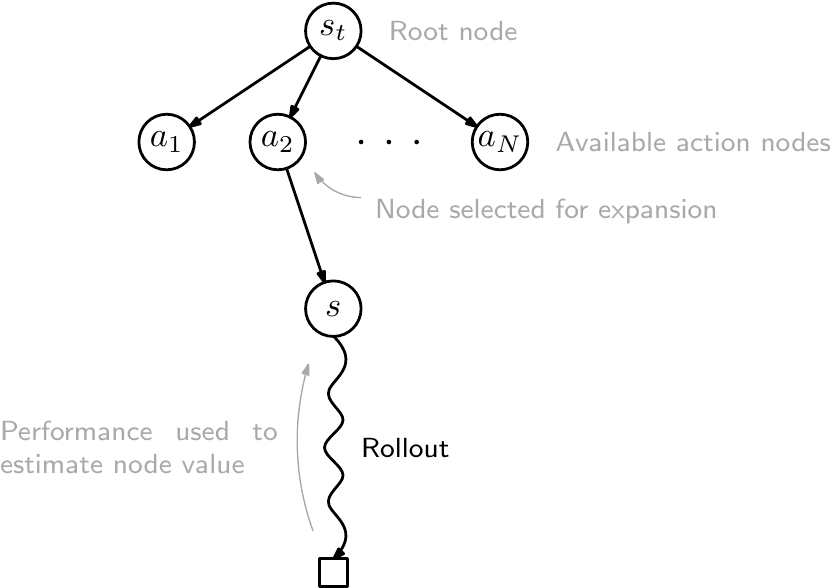}
\caption{Illustration of MCTS. An adequate exploration policy is used to determine the next node selected for expansion. Rollouts (i.e., sampled trajectories) from the expanded nodes are used to build estimates about how good the nodes are.}
\label{Fig:MCTS}
\end{figure}

In this paper, we extend the idea behind MCTS to control problems with a continuous action space. The main challenge lies in ensuring efficient exploration. Standard MCTS approaches assume that the decision-maker has only a finite set of possible actions available and typically use multi-armed bandit algorithms to guide exploration. Hence, our first contribution is a novel multi-armed bandit algorithm for continuous action spaces.

Several works have addressed multi-armed bandits with continuous action spaces \cite{munos14ftml}. Most such approaches consider an action space that corresponds to a hypercube in $\IR^P$, i.e., a set
\begin{equation*}
\X=\prod_{p=1}^P[a_p,b_p],\qquad\text{for $a_p,b_p\in\IR^P$.}	
\end{equation*}

Multi-armed bandit problems in continuous action settings generally apply some form of {\em hierarchical partitioning} \cite{munos14ftml}. Hierarchical partitioning successively splits $\X$ into smaller hypercubes in such a way that the resulting partition can be represented as a tree (see Fig.~\ref{Fig:HP}). Examples of such algorithms include {\em deterministic optimistic optimization} (DOO) \cite{munos11nips}, {\em hierarchical optimistic optimization} (HOO) \cite{bubeck09nips}, {\em stochastic simultaneous optimistic optimization} (StoSOO) \cite{munos14ftml}, {\em truncated hierarchical optimistic optimization} (T-HOO) \cite{bubeck11jmlr}, or {\em polynomial hierarchical optimistic optimization} (Poly-HOO) \cite{mao20nips}. 

\begin{figure}[!tb]
\centering
\includegraphics{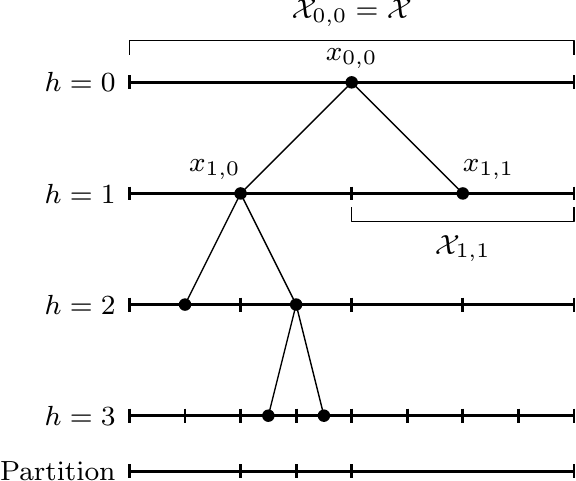}
\caption{Hierarchical partitioning. The set $\X$ is successively split in uniform subsets, and the resulting partition can be represented as a $K$-ary tree. The sets in level $h$ are referred as ``cells'' and denoted as $\X_{h,i}$, with $0\leq i\leq K^h-1$. Each set $\X_{h,i}$ is represented by a specific element $x_{h,i}$, referred as the {\em center} of $\X_{h,i}$.}
\label{Fig:HP}
\end{figure}

All the methods above also build on the principle of {\em optimism in the face of uncertainty}, widely used in the multi-armed bandit literature \cite{auer02ml} and MCTS. The most adequate for the settings considered in this paper are HOO and its truncated and polynomial variants. Unlike DOO, HOO admits stochastic rewards. Unlike StoSOO, it is an anytime algorithm, a desirable property when used within MCTS. Unfortunately, HOO exhibits quadratic runtime, which is inconvenient for MCTS. Poly-HOO and T-HOO are faster, but exhibit a worse performance in terms of cumulative regret. 

This paper contributes LD-HOO, a modified version of HOO that is computationally faster but retains HOO's guarantees in terms of cumulative regret. We also contribute an analysis of LD-HOO in terms of cumulative regret. Finally, we propose LD-HOOT (LD-HOO for trees), a novel MCTS algorithm for problems with continuous action spaces that uses LD-HOO for efficient exploration. We illustrate the application of LD-HOOT in several well-known optimal control problems.


\section{Background}%
\label{Sec:Background}

This section provides an overview of key concepts regarding the decision problems addressed in this paper and sets the notation used throughout.

\subsection{The (stochastic) multi-armed bandit problem}

We consider the problem of a decision maker that must, at each time step $t=1,2,\ldots$, select a (possibly random) action $X_t$ from a set $\X$ of possible actions. Upon selecting an action $X_t=x$, the decision maker receives a (possibly random) reward $Y_t$ with mean value $f(x)$. The process then repeats. 

We assume that $Y_t\in[0,1]$ and the distribution governing $Y_t$ is stationary and independent of the past given the action $X_t$. The function $f:\X\to\IR$ is unknown to the decision maker. The goal of the decision maker is to select the actions $X_t,t=1,2,\ldots$, to maximize the total reward received. 

A {\em policy} is a sequence $\pi=\set{\pi_t,\;t=1,2,\ldots}$, where each $\pi_t$ maps the history of the agent up to time step $t$ to a distribution over $\X$. The {\em cumulative regret} of a policy $\pi$ at time step $n$ is defined as
\begin{equation*}
\hat{R}_n=nf(x^*)-\sum_{t=1}^nY_t,
\end{equation*}
where $x^*=\argmax_{x\in\X}f(x)$, and the {\em cumulative pseudo-regret} is
\begin{equation*}
R_n=nf(x^*)-\sum_{t=1}^nf(x_t).
\end{equation*}
The cumulative regret measures how much the agent's policy lost in terms of cumulative reward compared with the optimal action. 

\subsection{Markov decision problems}

Markov decision problems (MDPs) extend multi-armed bandit problems by considering that the reward at each time step $t$, $Y_t$, is no longer independent of the past given only the action $X_t$. In an MDP, a decision-maker observes, at each time step $t=1,2,\ldots$, the {\em state} of a Markov chain $\set{S_t,t=1,2,\ldots}$, where each $S_t$ takes values in some set $\S$. It must then select a (possibly random) action $X_t$ from a set $\X$ of possible actions. Upon selecting an action $X_t=x$ when the chain is in state $S_t=s$, the decision maker receives a (possibly random) reward $Y_t$ with mean value $f(s,x)$. The Markov chain then transitions to a new state $S_{t+1}=s'$ according to the transition probabilities
\begin{equation*}
P(s,x,s')=\PP{S_{t+1}=s'\mid S_t=s,X_t=x}.
\end{equation*}

As before, we assume that the distribution governing $Y_t$ is stationary and independent of the past given $X_t$ and $S_t$. As before, the function $f$ is unknown to the learner. The goal of the decision maker is to select the actions $X_t,t=1,\ldots$, to maximize the value
\begin{equation}\label{Eq:TDR}
J=\EE{\sum_{t=1}^\infty\gamma^{t-1}f(S_t,X_t)},
\end{equation}
where $\gamma$ is a discount factor. MDPs are stochastic optimal control problems and find many applications in artificial intelligence, operations research and optimal control \cite{puterman05}.


\section{Limited Depth Hierarchical Optimistic Optimization}%
\label{Sec:LD-HOO}

This section introduces our first contribution, the LD-HOO algorithm. In terms of structure, LD-HOO is similar to HOO \cite{bubeck09nips}. LD-HOO also uses hierarchical partitioning to successively build a binary tree like the one depicted in Fig.~\ref{Fig:HP}. The key difference is that, in LD-HOO, the depth of the tree depth is limited to a maximum depth, $\hat{H}$. By limiting the tree's depth, we introduce significant computational savings without sacrificing performance.

At a very high-level, LD-HOO proceeds as follows. At each time step $t$, LD-HOO navigates the current partition tree, $\T_t$, from the root to some leaf node $X_{h_t,i_t}$ using an ``optimistic'' path (we make this notion precise ahead). It then selects an arbitrary action $x_t\in X_{h_t,i_t}$ and observes the corresponding reward, $y_t$. If $h_t<\hat{H}$ the node $X_{h_t,i_t}$ is expanded---i.e., its children nodes are added to the tree. Finally, the value of the nodes in the tree is updated accordingly.

To properly formalize the LD-HOO algorithm, we introduce some auxiliary notation. We denote by $\T_t$  the partition tree at time step $t$. The sets at each level $h$ of the tree are generally referred as ``cells'', and denoted as $\X_{h,i}$, with $0\leq i\leq 2^h-1$. Each cell $\X_{h,i}$ originates, in level $h+1$, two children cells, $\X_{h+1,2i-1}$ and $\X_{h+1,2i}$. We denote by $\C(h,i)$ the children of $\X_{h,i}$, i.e.,
\begin{equation*}
\C(h,i)=\set{(h+1,2i-1),(h+1,2i)},
\end{equation*}
and by
\begin{equation*}
T_{h,i}(t)=\sum_{\tau=1}^t\II{X_\tau\in\X_{h,i}}	
\end{equation*}
the number of times that an action from $\X_{h,i}$ was selected up to time step $t$. The estimated reward associated with a node $(h,i)$ at time step $t$ can be computed as
\begin{equation*}
\hat{\mu}_{h,i}(t)=\frac{1}{T_{h,i}(t)}\sum_{\tau=1}^ty_\tau\II{X_\tau\in\X_{h,i}}.
\end{equation*}
Finally, we write $\L(\T_t)$ to denote the set of leaf nodes of the tree $\T_t$.

The ``optimism in the face of uncertainty'' principle comes into play by associating, with each node $(h,i)$ in the tree, an upper confidence bound $b_{h,i}$ that is used to guide the traversal of the tree. The computation of this upper confidence bound relies on two key assumptions. 

The first assumption concerns the existence of a semi-metric in $\X$.

\begin{assumption}\label{Ass:Semi-metric}
The action space $\X$ is equipped with a semi-metric ${\ell\!:\X\!\times\!\X\!\to\!\IR}$. 	
\end{assumption}
Recall that $\ell$ is a semi-metric if, for all $x,x'\in\X$, $\ell(x,x')\geq 0$, $\ell(x,x')=\ell(x',x)$ and $\ell(x,x')=0$ if and only if $x=x'$. We can now define the diameter of a set $U\subset\X$ as 
\begin{equation*}
\diam(U)\defin\sup_{x,x'\in U}\ell(x,x').	
\end{equation*}
An $\ell$-open ball centered around $x\in\X$ with radius $\eps$ is defined as
\begin{equation*}
\B(x,\eps)\defin\set{x'\in\X\mid \ell(x,x')<\eps}.
\end{equation*}

The second assumption is concerned with hierarchical partitioning. 

\begin{assumption}\label{Ass:Hierarchy}
There are constants $\nu_1,\nu_2>0$ and $0<\rho<1$ such that, for all $h\geq 0$ and all $i\in\set{0,\ldots,2^h-1}$, there is $x_{h,i}\in\X_{h,i}$ such that
\begin{gather*}
\diam(\X_{h,i})\leq\nu_1\rho^h,\\
\B(x_{h,i},\nu_2\rho^h)\subset\X_{h,i},\\
\B(x_{h,i},\nu_2\rho^h)\cap\B(x_{h,j},\nu_2\rho^h)=\emptyset,\qquad\text{for $i\neq j$.}
\end{gather*}
\end{assumption}

Assumption~\ref{Ass:Hierarchy} ensures that the construction of the hierarchical partition and the semi-metric $\ell$ are such that the diameter of the cells in level $h$ decreases uniformly, and a suitable center can be found for each cell. The choice of partitioning method thus suggests a semi-metric, while the choice of $\ell$ limits the way that the partition can be done \cite{bubeck11jmlr}.

\vspace{2ex}

From the above assumptions we can now compute an upper confidence bound for the values of $f$ in a cell $(h,i)$ of the tree at time step $t$ as
\begin{equation}\label{Eq:Bound-u}
u_{h,i}(t)=\hat{\mu}_{h,i}(t)+\sqrt{\frac{2\log t}{T_{h,i}(t)}}+\nu_1\rho^h,
\end{equation}
where $\nu_1$ and $\rho$ are the constants from Assumption~\ref{Ass:Hierarchy}. It is interesting to note the similarity between the bound in \eqref{Eq:Bound-u} and the upper confidence bound in the well-known UCB1 algorithm~\cite{auer02ml}. The additional term, $\nu_1\rho^h$, accounts for the variation of $f$ within $\X_{h,i}$. 

The actual value used to traverse the tree is referred as the $b$-value, and is computed from the upper confidence bound in \eqref{Eq:Bound-u} as
\begin{equation}\label{Eq:b-value}
	b_{h,i}(t)=\min\set{u_{h,i}(t),\max_{c\in\C(h,i)}b_c(t)},
\end{equation}
where the $b$-value of a non-visited node is set to $+\infty$. These values can be computed recursively from the leaves to the root.

\begin{algorithm}[!htb]
    \caption{Limited Depth Hierarchical Optimistic Optimization (LD-HOO).}
    \label{Alg:LD-HOO}
\begin{algorithmic}[1]
    \Statex{\bf Initialization:}
    \State $\mathcal{T}_1 = \{(0,0)\}$
    \State $u_{0,0} \leftarrow +\infty$
    \State $b_{0,0} \leftarrow +\infty$
    \Statex
    \Statex{\bf Main loop:}
    \For{$t=1,\ldots,n$}
    	\State Compute $b$-values for all nodes in $\T_t$ using \eqref{Eq:b-value} 
	    \State Set $(h,i)=(0,0)$ \label{Step:Traverse-start} 
	    \Comment{Start in the root node}
        \While{$(h,i)\not\in\L(\T_t)$}
            \State $(h,i)\leftarrow\argmax_{c\in\C(h,i)}b_c(t)$ 
            \Comment{Traverse tree ``optimistically''}
        \EndWhile \label{Step:Traverse-end}
        \State $(h_t,i_t)\leftarrow(h,i)$ \label{Step:Node-selection}
        \Comment{Node selection}
        \State Sample action $x_t$ arbitrarily from $\X_{h_t,i_t}$
        \Comment{Action sampling}
        \State Observe reward $y_t$
        \If{$h<\hat{H}$} \label{Step:Expansion-start}
            \State $\mathcal{T}_{t+1}=\mathcal{T}_t\cup\mathcal{C}(h_t,i_t)$ 
            \Comment{Expand node $(h_t,i_t)$}
        \EndIf \label{Step:Expansion-end}
    \EndFor
    \Statex
    \Statex{\bf Final action selection:}
    \State {\bf Return $x_n\in\X_{h^*_n,i^*_n}$} \label{Step:Final-selection}
\end{algorithmic}
\end{algorithm}

LD-HOO is summarized in Algorithm~\ref{Alg:LD-HOO} for a horizon of $n$ steps. Lines~\ref{Step:Traverse-start}-\ref{Step:Traverse-end} correspond to the traversal of the partition tree: at node $(h,i)$, the algorithm moves to the child of $(h,i)$ with the largest $b$-value. The traversal stops at a leaf node $(h_t,i_t)$ (line~\ref{Step:Node-selection}) and an action is then selected from $\X_{h_t,i_t}$. Then, if $h_t$ is below the limit depth $\hat{H}$, the node $(h_t,i_t)$ is expanded and its children added to the tree (lines~\ref{Step:Expansion-start}-\ref{Step:Expansion-end}). Finally, upon completion, the algorithm returns an action $x_n$ from $\X_{h^*_n,i^*_n}$ where
\begin{equation*}
(h^*_n,i^*_n)=\argmax_{c\in\T_n}\hat{\mu}_c(n)	.
\end{equation*}

\subsection{Regret analysis for LD-HOO}

In this section we provide a regret analysis for LD-HOO. Our analysis closely follows that of Bubeck et al. \cite{bubeck11jmlr}, adapted to take into consideration the fact that LD-HOO considers limited-depth partition trees. 

Our analysis requires one additional assumption regarding the behavior of the mean payoff function, $f$, around its global maximizer, $x^*$.

\begin{assumption}\label{Ass:Smooth-payoff}
The mean payoff function, $f$, is such that, for all $x,x'\in\X$,
\begin{equation*}
f(x^*)-f(x')\leq f(x^*)-f(x)+\max\set{f(x^*)-f(x),\ell(x,x')}.
\end{equation*}
\end{assumption}
Assumption~\ref{Ass:Smooth-payoff} states that there is no sudden decrease in the mean payoff function around $x^*$. Also, setting $x=x^*$, Assumption~\ref{Ass:Smooth-payoff} actually establishes $f$ as locally Lipschitz around $x^*$ \cite{bubeck09nips}. The set of {\em $\eps$-optimal actions}, $\X_\eps$, is defined as
\begin{equation*}
\X_\eps\defin\set{x\in\X\mid f(x)\geq f(x^*)-\eps}.	
\end{equation*}
We have the following definition.

\begin{definition}[$\eps$-packing number]
Given a set $\X$ with a semi-metric $\ell$, the {$\eps$-packing number} of $\X$ with respect to $\ell$ is denoted as $\N(\X,\ell,\eps)$ and defined as the largest integer $K$ for which there are $x_1,\ldots,x_K\in\X$ such that $\B(x_k,\eps)\subset\X$ and $\B(x_k,\eps)\cap\B(x_{k'},\eps)=\emptyset$, for all $k,k'\in\set{1,\ldots,K}$.
\end{definition}

In other words, $\N(\X,\ell,\eps)$ is the maximum number of disjoint $\eps$-balls contained in $\X$. We now introduce the notion of near-optimality dimension, which describes the size of the sets $\X_{c\eps}$ as a function of $\eps$.

\begin{definition}[Near-optimality dimension] Given constants $c>0$ and $\eps_0>0$, the {\em $(c,\eps_0)$-near-optimality dimension} of $f$ with respect to $\ell$ is defined as
\begin{equation*}
d\defin\inf\set{\delta\geq 0\mid \exists C>0\text{ s.t. }\forall \eps\leq\eps_0,\N(\X_{c\eps},\ell,\eps)\leq C\eps^{-\delta}},	
\end{equation*}
with the convention that $\inf\emptyset=+\infty$. 
\end{definition}
In other words $d$ is the smallest value such that the maximum number of disjoint $\eps$-balls that ``fit'' in $\X_{c\eps}$ is no larger than a $C\eps^{-d}$. A small near-optimality dimension implies that there is a small number of neighborhoods where the value of $f$ is close to its maximum. 

We now introduce our main result.

\begin{theorem}\label{Theo:Regret}
Suppose that Assumptions~\ref{Ass:Semi-metric} through \ref{Ass:Smooth-payoff} are satisfied for semi-metric $\ell$ and constants $\nu_1$, $\nu_2$ and $\rho$. Let $d<+\infty$ be the $(4\nu_1/\nu_2,\nu_2)$-near-optimality dimension of $f$ with respect to $\ell$. Then, for all horizons $n>0$ and depth limit $\hat{H}>0$, the expected cumulative regret of {LD-HOO} verifies:
\begin{equation*}
\EE{R_n}\in O\left(\rho^{\hat{H}}n+\rho^{-\hat{H}(d+1)}\log n+\hat{H}\max\{\rho^{\hat{H}(1-d)},\hat{H}\}\right).
\end{equation*}
\end{theorem}
Due to space limitations, we provide only a sketch of the proof, and refer to the supplementary material for a complete proof.

\begin{proof}[Proof sketch]
The proof proceeds in four steps. We start by providing a bound on the number of times that each ``suboptimal'' node in the tree is selected. In a second step we break down the value of $\EE{R_n}$ into three components. In the third step we show that two of the regret components are naturally bounded as a consequence of our assumptions, and the third depends on the number of visits to the nodes in the tree (which was bounded in the first step). The fourth step wraps the proof by combining the intermediate results into a final regret bound.

\paragraph*{First step.} For a node $(h,i)$ in $\T_n$, let $f_{h,i}^*=\sup_{x\in\X_{h,i}}f(x)$, and define the set 
\begin{equation*}
\I_h=\set{(h,i)\mid f^*_{h,i}\geq f(x^*)-2\nu_1\rho^h},
\end{equation*}
i.e., $\I_h$ is the set of $2\nu_1\rho^h$-optimal nodes at depth $h$. Let $\J_h$ be the set of nodes at depth $h$ that are not in $\I_h$ but whose parents are in $\I_{h-1}$. We can bound the expected number of times that an action from a node $(h,i)\in\J_h$ is selected up to time step $n$ as 
\begin{equation}\label{Eq:Bound-Jh}
    \EE{T_{h,i}(n)}\leq\frac{8\log n}{\nu_1^2\rho^{2h}}+2(\hat{H}+1).
\end{equation}

Furthermore, we can bound the number of nodes in $\I_h$ as
\begin{equation}\label{Eq:Ih-cardinality}
\abs{\I_h}\leq C(\nu_2\rho^h)^{-d}.
\end{equation}

\paragraph*{Second step.} We partition the nodes of $\T_n$ into three disjoint sets, $\T^1$, $\T^2$, and $\T^3$, such that $\T^1\cup\T^2\cup\T^3=\T_n$. In particular, 
\begin{itemize}
\item $\T^1$ is the set of nodes at depth $\hat{H}$ that are $2\nu_1\rho^{\hat{H}}$-optimal, i.e., $\T^1=\I_{\hat{H}}$.
\item $\T^2$ is the sets of all $2\nu_1\rho^h$-optimal nodes, for all depths $h\neq\hat{H}$, i.e., $\T^2=\cup_{h=0}^{\hat{H}-1}\I_h$.
\item $\T^3$ is the set set of all remaining nodes, i.e., $\cup_{h=1}^{\hat{H}}\J_h$ and corresponding descendant nodes.
\end{itemize}
We now define the regret $R_{n,k}$ of playing actions in $\T^k, k=1,2,3$, as 
\begin{equation}
    R_{n,k} = \sum_{t=1}^n(f(x^*)-f(X_t))\II{(H_t,I_t)\in\T^k}
\end{equation}
where $f(X_t)$ is the expected reward collected at time step $t$ given the action $X_t$ and $(H_t,I_t)$ is the (random) node selected in that same time step. It follows that
\begin{equation}\label{Eq:Regret-decomposed}
\EE{R_n}=\EE{R_{n,1}}+\EE{R_{n,2}}+\EE{R_{n,3}}.
\end{equation}

\paragraph*{Third step.} We now bound each of the terms in \eqref{Eq:Regret-decomposed} to establish a bound for the total cumulative regret after $n$ rounds.

Starting with $R_{n,1}$, in the worst case we have that
\begin{equation}\label{Eq:Bound-R1}
\EE{R_{n,1}}\leq4n\nu_1\rho^{\hat{H}}.
\end{equation}

In turn, the nodes $(h,i)\in\T^2$ are $2\nu_1\rho^h$-optimal. This means that
\begin{equation*}
\EE{R_{n,2}}\leq\sum_{t=1}^n4\nu_1\rho^{H_t}\II{(H_t,I_t)\in\T^2}
\end{equation*}
However, each node $(h,i)\in\T^2$ is selected only once and then expanded---only nodes at depth $\hat{H}$ can be selected multiple times. As such, we can use the bound in \eqref{Eq:Ih-cardinality} to get
\begin{equation}\label{Eq:Bound-R2}
\EE{R_{n,2}}
  \leq\sum_{h=0}^{\hat{H}-1}4\nu_1\rho^h\abs{\I_h}
  \leq4C\nu_1\nu_2^{-d}\sum_{h=0}^{\hat{H}-1}\rho^{h(1-d)}
\end{equation}

Finally, let us consider the nodes in $\T^3$. Each node $(h,i)\in\J_h$ is the child of a node in $\I_{h-1}$ and, as such, $f(x^*)-f^*_{h,i}\leq4\nu_1\rho^{h-1}$. Additionally, $\abs{\J_h}\leq 2\abs{\I_{h-1}}$, with equality if $\I_{h-1}$ was fully expanded and all the direct children of the nodes $(h-1,i)\in\I_h$ are in $\T_n$. 
 
Using \eqref{Eq:Bound-Jh}, we can now bound $\EE{R_{n,3}}$ as

\begin{align}
\nonumber%
\EE{R_{n,3}}
  &\leq\sum_{h=1}^{\hat{H}}4\nu_1\rho^{h-1}\sum_{(h,i)\in\J_h}\EE{T_{h,i}(n)}\\
\nonumber%
  &\leq\sum_{h=1}^{\hat{H}}4\nu_1\rho^{h-1}\abs{\J_h}\left(\frac{8\log n}{\nu_1^2\rho^{2h}}+2(\hat{H}+1)\right)\\
\nonumber%
  &\leq\sum_{h=1}^{\hat{H}}8\nu_1\rho^{h-1}\abs{\I_{h-1}}\left(\frac{8\log n}{\nu_1^2\rho^{2h}}+2(\hat{H}+1)\right)\\ 
\label{Eq:Bound-R3}%
  &\leq8C\nu_1\nu_2^{-d}\sum_{h=0}^{\hat{H}-1}\rho^{h(1-d)}\left(\frac{8\log n}{\nu_1^2\rho^{2h}}+2(\hat{H}+1)\right).
\end{align}

\paragraph*{Fourth step.} We now combine the bounds~\eqref{Eq:Bound-R1}, \eqref{Eq:Bound-R2} and \eqref{Eq:Bound-R3} to get
\begin{equation}\label{Eq:Combined-bound}
\EE{R_n}
  \leq4n\nu_1\rho^{\hat{H}}+4C\nu_1\nu_2^{-d}\sum_{h=0}^{\hat{H}-1}\rho^{h(1-d)}\left(\frac{8\log n}{\nu_1^2\rho^{2h}}+4\hat{H}+5\right)
\end{equation}
Rewriting \eqref{Eq:Combined-bound} using asymptotic notation, we finally get
\begin{equation*}
\EE{R_n}\in O\left(n\rho^{\hat{H}}+\log n\rho^{-\hat{H}(1+d)}+\hat{H}\max\{\rho^{\hat{H}(1-d)},\hat{H}\}\right),
\end{equation*}
and the proof is complete.\qed
\end{proof}

We get the following corollary to Theorem~\ref{Theo:Regret}.

\begin{corollary}\label{Cor:LDHOO-bound}
Under the conditions of Theorem~\ref{Theo:Regret}, selecting $\hat{H}$ so that
\begin{equation*}
\rho^{\hat{H}}=c\left(\frac{n}{\log n}\right)^{-\frac{1}{d+2}},
\end{equation*}
for some constant $c>0$, we get
\begin{equation}\label{Eq:Asymptotic-regret}
\EE{R_n}\in O\left(n^{\frac{d+1}{d+2}}(\log n)^{\frac{1}{d+2}}\right).
\end{equation}
\end{corollary}

Corollary~\ref{Cor:LDHOO-bound} establishes that, through an adequate choice of the maximum depth $\hat{H}$, the bound on the cumulative regret of LD-HOO matches that of HOO \cite{bubeck11jmlr}. In the continuation, we establish that both the running time and the size of the partition tree of LD-HOO are better than those of HOO.

\subsection{Time and space complexity of LD-HOO}

The number of nodes in the partition tree generated by LD-HOO is upper bounded by 
\begin{equation*}
\sum_{h=0}^{\hat{H}}2^h\in O(2^{\hat{H}}).	
\end{equation*}
 However, under the conditions of Corollary~\ref{Cor:LDHOO-bound}, $\hat{H}\in O(\log n)$. As such, the space complexity of LD-HOO is $O(n^{\log 2})$. In contrast, the number of nodes in a partition tree generated by HOO is $O(n)$. Thus, the space complexity of LD-HOO is asymptotically better than HOO's.

Similarly, the running time of LD-HOO is proportional to the number of updates to the $b$-values. After $n$ iterations, the total amount of updates is, at most, $n$ times the number of nodes in the tree. As such, the time complexity of LD-HOO is $O(n^{1+\log 2})$, against the $O(n^2)$ time of HOO \cite{bubeck11jmlr}.

\subsection{Comparison with other algorithms}

We conclude this section by performing an empirical comparison between LD-HOO and related algorithms---namely HOO \cite{bubeck09nips}, Poly-HOO \cite{mao20nips} and T-HOO \cite{bubeck11jmlr}. 

In our tests, we use a mean reward function $f(x)=\frac{1}{2}(\sin(13x)\sin(27x)+1)$. The reward received by the decision-maker at each time step is a perturbed version of $f$: upon selecting an action $X_t$, the decision-maker receives a reward $Y_t=f(X_t)+\eps_t$, where $\eps$ is drawn from a Gaussian distribution with zero mean and a standard deviation of $0.05$. We set $\nu_1=1$ and $\rho=0.25$ for all algorithms. For LD-HOO we set $\hat{H}=\lceil\log n\rceil$.%
\footnote{Even though our selected value for $\hat{H}$ does not verify the conditions in Corollary~\ref{Cor:LDHOO-bound}, it empirically performed well and, as such, was used in our experiments.}
We compare in Fig.~\ref{Fig:Results} the performance of the 4 methods, as we vary the horizon between $n=10$ and $n=1,000$. The results portrayed are averages of $10$ independent runs. 

\begin{figure}[!tb]
\centering
\begin{subfigure}[t]{0.49\columnwidth}
  \includegraphics[width=\textwidth]{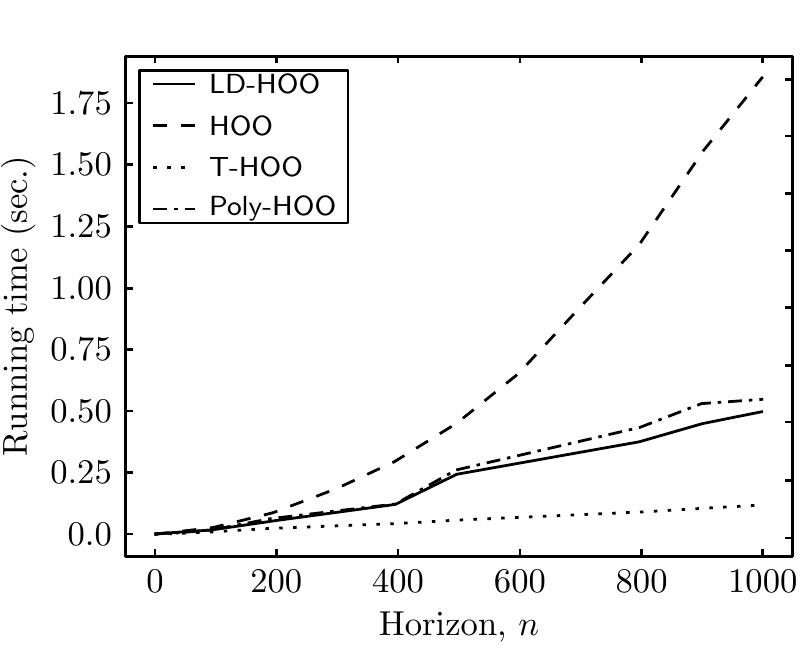}	
  \caption{Running time.}
  \label{Fig:Running-time}
\end{subfigure}\hfill
\begin{subfigure}[t]{0.49\columnwidth}
  \includegraphics[width=\textwidth]{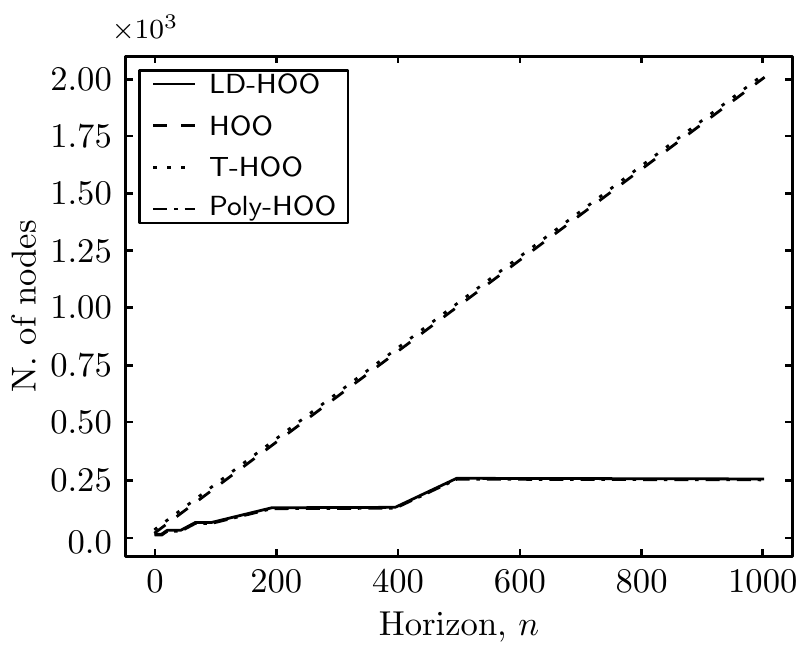}	
  \caption{Number of nodes in the partition tree.}
  \label{Fig:Nodes}
\end{subfigure}\\
\begin{subfigure}[t]{0.5\columnwidth}
  \includegraphics[width=\textwidth]{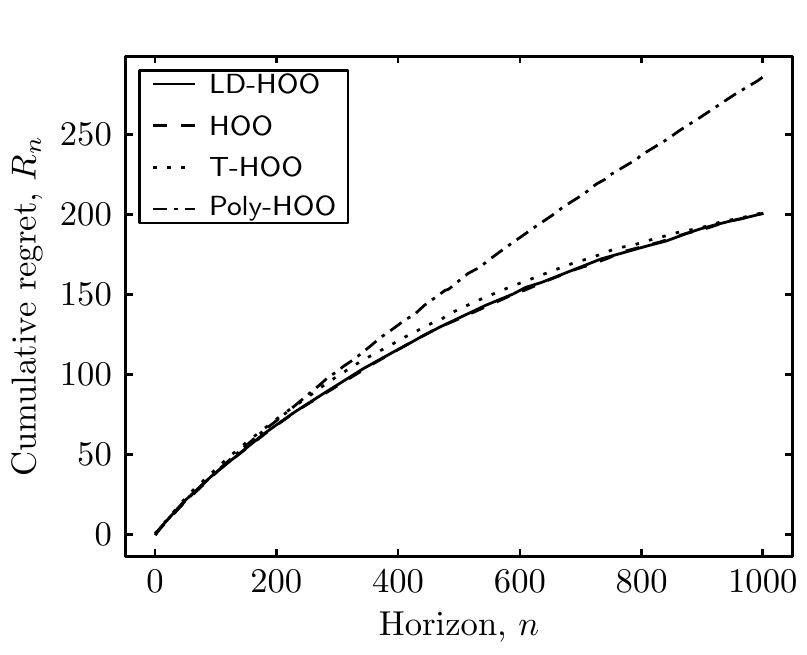}	
  \caption{Cumulative regret.}
  \label{Fig:Regret}
\end{subfigure}
  \caption{Performance of all HOO-based algorithms in terms of time, space and regret, as a function of the horizon, {$n$}.}
  \label{Fig:Results}
\end{figure}

Our results are in accordance with our theoretical analysis. As seen in Figs.~\ref{Fig:Running-time} and \ref{Fig:Nodes}, HOO is the most computationally demanding algorithm, both in terms of time and in terms of the number of nodes. LD-HOO is similar to Poly-HOO in terms of computational complexity but, as seen in Fig.~\ref{Fig:Regret}, it outperforms Poly-HOO in terms of cumulative regret. Finally, LD-HOO is similar to T-HOO in terms of regret; T-HOO exhibits a superior running time, but a worse number of nodes (the line of T-HOO in Fig.~\ref{Fig:Nodes} overlaps that of HOO).

The conclusion from our empirical results is that, overall, LD-HOO achieves similar or better performance than the other algorithms. The only exception is in terms of running time, where T-HOO is noticeably faster. However, since our ultimate goal is to use our algorithm in the context of MCTS, T-HOO is not an option,  since it is not an anytime algorithm \cite{bubeck11jmlr}.

\section{Monte Carlo Tree Search using LD-HOO}%
\label{Sec:LD-HOOT}

We are now in position to introduce LD-HOOT (or LD-HOO for trees), where we combine Monte Carlo tree search with LD-HOO. We consider Markov decision problems (MDPs) where, at each step, the agent/decision-maker observes the current state of the MDP, $S_t$, and must select an action, $X_t$, that maximizes the total discounted reward, as defined in \eqref{Eq:TDR}. Like standard MCTS algorithms, at each step $t$ LD-HOOT builds a tree representation of the MDP similar to that in Fig.~\ref{Fig:MCTS}. To traverse the tree, each state node $s$ is treated as a multi-armed bandit problem, and an action is selected using LD-HOO. When computation time is over, LD-HOOT returns the action selected by LD-HOO at the root node (line~\ref{Step:Final-selection} of Algorithm~\ref{Alg:LD-HOO}).

\begin{figure}[!tb]
\begin{subfigure}[b]{0.48\columnwidth}
\centering
\includegraphics[width=0.5\textwidth]{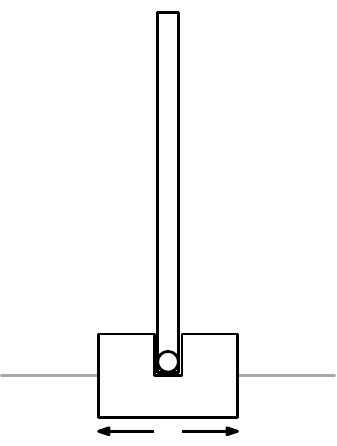}
\caption{The cart pole environment.}	
\end{subfigure} \hfill
\begin{subfigure}[b]{0.48\columnwidth}
\centering
\includegraphics[width=0.5\textwidth]{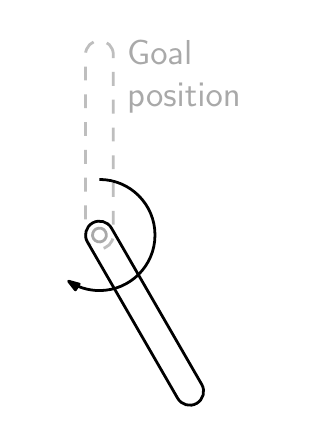}
\caption{The inverted pendulum environment.}	
\end{subfigure} \qquad
\caption{Illustration of the two environments used to evaluate LD-HOOT. The black arrows illustrate the control available: in the cart-pole scenario, the control is performed by applying a force to the cart. In the inverted pendulum, the control is performed by applying a force to the pendulum.}
\label{Fig:Environments}
\end{figure}

To assess the applicability of LD-HOOT, we tested the algorithm in two control problems from the literature. We used benchmark environments from OpenAI's Gym \cite{brockman16arxiv}, namely \texttt{CartPole-v1} and \texttt{Pendulum-v0}, as well as the modified version of cart-pole with increased gravity of Mao et al. \cite{mao20nips}. In the two cart-pole scenarios, the pole is initialized close to the equilibrium position, and the role of the agent is to control the cart to keep the pole (approximately) vertical. In the inverted pendulum scenario, the pendulum is initialized in an arbitrary position above the horizontal line with (possible) non-zero initial velocity, and the role of the agent is to balance the pole to the vertical position.

We used $\nu_1=4$, $\rho=0.25$ and allowed LD-HOOT to have a lookahead limit of $D=50$ actions. In all scenarios, we set $\hat{H}=\lceil\log n\rceil$. In the cart-pole scenario, the agent interacted with the environment for a total of $T=150$ time steps. At each time step, the action of the agent was computed using $n=100$ iterations of the planner. In contrast, in the inverted pendulum scenario, the agent interacted with the environment for a total of $T=100$ time steps. However, since the environment is more difficult, each action was computed from $n=100,400$ iterations of the planner to establish the impact of the increased $n$ in each of the algorithms. The number of actions executed in each environment permits a comparison with the benchmark results in Open AI Gym \cite{leaderboard}. We normalized the rewards to the interval $[0,1]$ to allow an easy comparison with the number of actions executed since, in this case, that is (potentially) the maximum cumulative reward that could be obtained.

\begin{table}[!tb]
\centering
\caption{Average sum of rewards obtained in the four test scenarios.}
\label{Table:Results-hoot}
\begin{tabular}{lcccc}
\toprule
& \sc Cart Pole & \sc Cart Pole IG & \multicolumn{2}{c}{ \sc Inverted Pendulum} \\ 
$n$ & 100 & 100 & 100 & 400 \\ 
\midrule
LD-HOOT & $150.0 \pm 0.0$ & $150.0 \pm 0.0$ & $82.46 \pm 10.40$ & $84.36 \pm  9.00$ \\ 
HOOT & $150.0 \pm 0.0$ & $150.0 \pm 0.0$ & $81.23 \pm 13.41$& $81.21 \pm 11.45$ \\ 
POLY-HOOT & $150.0 \pm 0.0$ & $150.0 \pm 0.0$ & $82.40 \pm 11.24$ & $84.56 \pm  8.79$ \\
\bottomrule
\end{tabular}
\end{table}

Table~\ref{Table:Results-hoot} compares the performance of LD-HOOT against that of HOOT \cite{mansley11icaps} and Poly-HOOT \cite{mao20nips}, the MCTS extensions of HOO and Poly-HOO, respectively. For each algorithm, we report the sum of rewards obtained averaged over $10$ independent trials, each with a different randomly selected initial state (but the same across algorithms). The only exception are the results obtained in the Inverted Pendulum with $n=100$, were the number of trials was $30$. The differences in variance observed with different $n$ in the Inverted Pendulum tests are likely due to the discrepancy in the number of trials and, consequently, the usage of different initial states. Even though just slightly, HOOT is the only of the 3 algorithms that performed worse when running for a longer period. On the other hand, both LD-HOOT and Poly-HOOT accrued rewards over $2\%$ larger in this scenario, thus evidencing a superior performance against their predecessor (HOOT). As seen in Table~\ref{Table:Results-hoot}, all three algorithms perform optimally in the two cart-pole scenarios (dubbed {\sc Cart Pole} and {\sc Cart Pole ig}). Even with increased gravity, all algorithms were able to maintain the pole approximately vertical. In contrast, in the inverted pendulum scenario LD-HOOT outperforms HOOT, and performs similarly to Poly-HOOT. It is also noteworthy that, according to Open AI's leaderboard \cite{leaderboard}, the best performing algorithm in this scenario obtained an average cumulative reward of $92.45$. Our results are close but, due to the large variance, it is not possible to assert whether the differences are significative or not.

\begin{table}[!tb]
\centering
\caption{Average running time (in seconds) required to plan a single action, for different values of $n$ and $\hat{H}$.}
\label{Table:Running-time}
\begin{tabular}{lccc}
\toprule
\multicolumn{1}{c}{$n\mid\hat{H}$} 
  & $100\mid 5$ & $400\mid 6$ & $1000\mid 7$ \\
\midrule
\sf LD-HOOT   & $1.14 \pm 0.07$ & $4.51 \pm 0.06$ & $11.88 \pm 0.05$\\
\sf HOOT      & $1.06 \pm 0.04$ & $4.72 \pm 0.01$ & $15.15 \pm 2.55$\\
\sf POLY-HOOT & $1.15 \pm 0.03$ & $4.90 \pm 0.21$ & $12.76 \pm 0.15$\\
\bottomrule
\end{tabular}
\end{table}

Table~\ref{Table:Running-time} reports the running time for the three algorithms for different values of $n$ and $\hat{H}$, when planning a single action. The results again correspond to averages over 10 independent trials. Our results confirm that LD-HOOT is faster than HOOT, which was expected since LD-HOO is also faster than HOO. When $n=100$ and $\hat{H}=5$, HOOT was marginally faster, but given the small times considered it is possible that the differences may be due to other processes that are being run at the same time. LD-HOOT is also faster than Poly-HOOT in all settings considered, but differences were smaller.

Summarizing, LD-HOOT outperforms HOOT in both metrics considered, and is slightly faster than Poly-HOOT, although attaining similar cumulative rewards.

\section{Conclusion}%
\label{Sec:Conclusions}

In this work, we contributed LD-HOO, a novel algorithm for stochastic multi-armed bandit problems with a continuous action spaces. We provided a regret analysis for LD-HOO and establish that, under mild conditions, LD-HOO exhibits an asymptotic regret bound similar to that of other algorithms in the literature, while incurring in smaller computational costs. Our analysis was then complemented by an empirical study, in which we show that LD-HOO generally outperforms similar algorithms. 

We then propose LD-HOOT, a Monte Carlo tree search algorithm that uses LD-HOO to drive the expansion of the search tree. We tested LD-HOOT in several optimal control problems from the literature, and the results show that LD-HOO can be successfully used within MCTS to address optimal control problems with continuous actions spaces. Additionally, our results show that LD-HOOT is competitive with existing algorithms in terms of performance, while being computationally more efficient. 

An interesting line for future research is to build on the results presented herein---namely Theorem~\ref{Theo:Regret}---to establish the consistence of LD-HOOT in the sense that the probability of selecting the optimal action can be made to converge to 1 as the number of samples grows to infinity. Such results exist for MCTS with finite action spaces \cite{kocsis06ecml}, and we conjecture that a similar line of proof could be used to establish similar guarantees for LD-HOOT.

\section{Acknowledgments}

This work was partially supported by national funds through FCT, Fundação para a Ciência e a Tecnologia, under project UIDB/50021/2020.

\bibliographystyle{splncs04}
\bibliography{biblio.bib}

\newpage

\setcounter{page}{1}

\appendix

\title{Limited depth bandit-based strategy for Monte Carlo planning in continuous action spaces}
\titlerunning{Suplementary material}

\author{\sf Supplementary material for paper n.\ 923}
\institute{}

\maketitle

\begin{abstract}
This appendix includes a detailed proof for Theorem~\ref{Theo:Regret}, establishing an asymptotic regret bound for LD-HOO.
\end{abstract}

\section{Proof of Theorem 1}

The proof follows the structure of the proof sketch provided in the manuscript, and is also based on the analysis of HOO found in the work of Bubeck et al \cite{bubeck11jmlr}. Before dwelling into the proof, we introduce some auxiliary results. 

\subsection{Auxiliary results}

We define the regret of a node $(h,i)\in\T_n$ as 
\begin{equation*}
r_{h,i}=f(x^*)-\sup_{x\in\X_{h,i}}f(x).
\end{equation*}
Fix an optimal path $(0,0)\to(1,i^*_1)\to\ldots\to(\hat{H},i^*_{\hat{H}})$, where $(h,i^*_h)$ is a child of $(h-1,i_{h-1}^*)$ and $i^*_h$ is such that $r_{h,i^*_h}=0$. Define the set
\begin{equation*}
\hat{\C}(h,i)=\set{c\in\T_n\mid\X_c\subset\X_{h,i}},
\end{equation*}
i.e., $\hat{\C}(h,i)$ is the set containing node $(h,i)$ and all its descendants. The next lemma bounds the number of times that a suboptimal node $(h,i)$ is selected up to iteration $n$. Having a bound for the expected number of times that a suboptimal node can be played is crucial to bound the expected cumulative regret.

\begin{lemma}\label{lemma:expected_t_hi_1}
Let $(h_0,i_0)$ be a suboptimal node (i.e., a node not in the path above), and let $k$ be the largest integer such that $0\leq k\leq h_0-1$ and $(k,i^*_k)$ is in the path from $(0,0)$ to $(h_0,i_0)$.\footnote{The extreme case is $k=0$.} Then for all integers $z\geq0$, we have

\begin{multline*}
\EE{T_{h_0,i_0}(n)}\leq z+\sum_{t=z+1}^n\mathbb{P}\Big[\big(u_{h,i^*_h}(t)\leq f(x^*) \text{ for some } h\in\{k+1,\ldots,\hat{H}\}\big) \\ 
    \text{ or } \big(T_{h_0,i_0}(t)>z \text{ and } u_{h_0,i_0}(t)>f(x^*)\big)\Big].
\end{multline*}
\end{lemma}
\begin{proof}
Let $(h_t,i_t)$ denote the node selected by LD-HOO at some time step $t$. If $(h_t,i_t)\in\hat{\C}(h_0,i_0)$, then at depth $k$ there is a node $(k+1,i)\in\C(k,i^*_k)$ which, by definition of $k$, is different from $(k+1,i^*_{k+1})$.

By construction, the $b$-value of a node is greater than or equal to that of its parent node. Hence, since node $(k+1,i)$ was selected instead of $(k+1,i^*_{k+1})$, then $b_{k+1,i^*_{k+1}}(t)\leq b_{k+1,i}(t)\leq b_{h_0,i_0}(t)$. This implies, again by construction, that $b_{h_0,i_0}(t)\leq u_{h_0,i_0}(t)$ and $b_{k+1,i^*_{k+1}}(t)\leq u_{h_0,i_0}(t)$. Putting everything together, we have
\begin{equation}\label{eq:lemma_htin}
\begin{split}
(h_t,i_t)\in\hat{\C}(h_0,i_0)
  &\Rightarrow u_{h_0,i_0}(t)\geq b_{k+1,i^*_{k+1}}(t) \\
  &\Rightarrow \big(u_{h_0,i_0}(t)>f(x^*) \text { or } b_{k+1,i^*_{k+1}}(t)\leq f(x^*)\big).
\end{split}
\end{equation}
To see why the last implication holds, suppose that $u_{h_0,i_0}(t)\geq b_{k+1,i^*_{k+1}}(t)$. Then, if $u_{h_0,i_0}(t)\leq f(x^*)$, it follows that $b_{k+1,i^*_{k+1}}(t)\leq u_{h_0,i_0}(t)\leq f(x^*)$, and the conclusion follows. 

We also have that
\begin{equation}\label{eq:lemma_bk_in_bk1}
b_{k+1,i^*_{k+1}}(t)\leq f(x^*)
  \Rightarrow \big(u_{k+1,i^*_{k+1}}(t)\leq f(x^*) \text{ or } b_{k+2,i^*_{k+2}}(t)\leq f(x^*)\big).
\end{equation}
To see why the implication in \eqref{eq:lemma_bk_in_bk1} holds, suppose that $b_{k+1,i^*_{k+1}}\leq f(x^*)$ holds. Then, if $u_{k+1,i^*_{k+1}}(t)> f(x^*)$, the child of $(k+1,i^*_{k+1})$ with maximum $b$-value, say $(k+2,j)$, satisfies $b_{k+2,j}(t)\leq f(x^*)$ since, by construction,
\begin{equation*}
	b_{k+1,i^*_{k+1}}(t)=\min\set{u_{k+1,i^*_{k+1}}(t),b_{k+2,j}(t)}.
\end{equation*}
Then, we must have $b_{k+2,i^*_{k+2}}(t)\leq f(x^*)$ and the conclusion follows. 

Finally, since LD-HOO does not expand nodes beyond depth $\hat{H}$, we have that
\begin{equation}\label{eq:lemma_bH}
b_{\hat{H},i^*_{\hat{H}}}(t)=u_{\hat{H},i^*_{\hat{H}}}(t).
\end{equation}
Unfolding \eqref{eq:lemma_bk_in_bk1} and combining with \eqref{eq:lemma_bH} yields
\begin{align*}
\lefteqn{b_{k+1,i^*_{k+1}}(t)\leq f(x^*)}\\
    &\Rightarrow\big(u_{k+1,i^*_{k+1}}(t)\leq f(x^*) \text{ or } b_{k+2,i^*_{k+2}}(t)\leq f(x^*)\big)\\
    &\Rightarrow\big(u_{k+1,i^*_{k+1}}(t)\leq f(x^*) \text{ or } u_{k+2,i^*_{k+2}}(t)\leq f(x^*) \text{ or } b_{k+3,i^*_{k+3}}(t)\leq f(x^*)\big)\\
    & \vdots \\
    &\Rightarrow \big(u_{k+1,i^*_{k+1}}(t)\leq f(x^*) \text{ or } u_{k+2,i^*_{k+2}}(t)\leq f(x^*) \text{ or } \ldots \text{ or }u_{\hat{H},i^*_{\hat{H}}}(t)\leq f(x^*)\big).
\end{align*}
Combining with \eqref{eq:lemma_htin}, we finally get
\begin{multline}\label{eq:lemma_total_inclusion}
(h_t,i_t)\in\hat{\mathcal{C}}(h_0,i_0)\\
  \Rightarrow \big(u_{h_0,i_0}(t)>f(x^*) \text{ or }u_{k+1,i^*_{k+1}}(t)\leq f(x^*) \text{ or }  \ldots \text{ or }u_{\hat{H},i^*_{\hat{H}}}(t)\leq f(x^*).
\end{multline}

For any integer $z\geq 0$, 
\begin{align*}
\lefteqn{T_{h_0,i_0}(n)}\\
  &=\sum_{t=1}^n\II{(h_t,i_t)\in\hat{\C}(h_0,i_0)}\\
  &=\sum_{t=1}^n\II{(h_t,i_t)\in\hat{\C}(h_0,i_0),T_{h_0,i_0}(t)\leq z}+\sum_{t=1}^n\II{(h_t,i_t)\in\hat{\C}(h_0,i_0),T_{h_0,i_0}(t)>z}\\
  &\leq z+\sum_{t=z+1}^n\II{(h_t,i_t)\in\hat{\C}(h_0,i_0),T_{h_0,i_0}(t)>z}.
\end{align*}
Taking the expectation and replacing \eqref{eq:lemma_total_inclusion}, we finally get
\begin{align*}
\lefteqn{\EE{T_{h_0,i_0}(n)}}\\
  &\leq z+\sum_{t=z+1}^n\PP{(h_t,i_t)\in\hat{\C}(h_0,i_0),T_{h_0,i_0}(t)>z}\\
  &\leq z+\sum_{t=z+1}^n\mathbb{P}\Big[T_{h_0,i_0}(t)>z \text{ and } \big(u_{h,i^*_h}(t)\leq f(x^*) \text{ for some } h\in\{k+1,\ldots,\hat{H}\}\big) \\ 
  &  \hspace{5.5cm}\text{ or } \big(T_{h_0,i_0}(t)>z \text{ and } u_{h_0,i_0}(t)>f(x^*)\big)\Big]\\
  &\leq z+\sum_{t=z+1}^n\mathbb{P}\Big[\big(u_{h,i^*_h}(t)\leq f(x^*) \text{ for some } h\in\{k+1,\ldots,\hat{H}\}\big) \\ 
  &  \hspace{5.5cm}\text{ or } \big(T_{h_0,i_0}(t)>z \text{ and } u_{h_0,i_0}(t)>f(x^*)\big)\Big].
\end{align*}
The proof is thus complete.
\end{proof}

The next lemmata are established by Bubeck et al. \cite{bubeck11jmlr} and are provided without proof.

\begin{lemma}\label{lemma:p_optimal}
If Assumptions~\ref{Ass:Semi-metric} through \ref{Ass:Smooth-payoff} hold then, for all optimal nodes $(h,i)$ and integers $t>0$,
\begin{equation*}
    \PP{u_{h,i}(t)\leq f(x^*)}\leq\frac{1}{t^3}.
\end{equation*}
\end{lemma}

\begin{lemma}\label{lemma:p_suboptimal}
Suppose that Assumptions~\ref{Ass:Semi-metric} through \ref{Ass:Smooth-payoff} hold. Then, for all integers $t\leq n$, all suboptimal nodes $(h,i)$ such that $r_{h,i}>\nu_1\rho^h$, and for all integers $z$ such that
\begin{equation*}
    z\geq\frac{8\log n}{(r_{h,i}-\nu_1\rho^h)^2},
\end{equation*}
it holds that
\begin{equation*}
\PP{u_{h,i}(t)>f(x^*) \text{ and } T_{h,i}(t)>z}\leq\frac{t}{n^4}.
\end{equation*}
\end{lemma}

\begin{lemma}\label{lemma:subopt_factor}
Suppose that Assumptions~\ref{Ass:Semi-metric} through \ref{Ass:Smooth-payoff} hold. If, for a node $(h,i)$, $r_{h,i}\leq c\nu_1\rho^h$, then
\begin{equation}
\X_{h,i}\subset\X_{\max\{2c,c+1\}\nu_1\rho^h}.
\end{equation}
\end{lemma}

The next result provides an upper bound for the number of visits to a suboptimal node $(h,i)$.

\begin{lemma}\label{lemma:expected_t_hi_2}
If Assumptions~\ref{Ass:Semi-metric} through \ref{Ass:Smooth-payoff} hold then, for all suboptimal nodes $(h,i)$ with $r_{h,i}>\nu_1\rho^h$ and $n\geq 1$,
\begin{equation*}
\EE{T_{h,i}(n)}\leq\frac{8\log n}{(r_{h,i}-\nu_1\rho^h)^2}+2(\hat{H}+1).
\end{equation*}
\end{lemma}
\begin{proof}
Let $(h_0,i_0)$ denote an arbitrary suboptimal node. Taking 
\begin{equation*}
z=\left\lceil\frac{8\log n}{(r_{h_0,i_0}-\nu_1\rho^{h_0})^2}\right\rceil
\end{equation*}
in Lemma~\ref{lemma:expected_t_hi_1}, we get 
\begin{multline*}
\EE{T_{h_0,i_0}(n)}
  \leq\frac{8\log n}{(r_{h,i}-\nu_1\rho^h)^2}+1\\
    +\sum_{t=z+1}^n\Big(\PP{T_{h_0,i_0}(t)>z \text{ and } u_{h_0,i_0}(t)>f(x^*)}
    +\sum_{h=1}^{\hat{H}}\PP{u_{h,i^*_h}(t)\leq f(x^*)}\Big).
\end{multline*}

Using Lemmas~\ref{lemma:p_optimal} and \ref{lemma:p_suboptimal}, 
\begin{align*}
\EE{T_{h_0,i_0}(n)}
    &\leq\frac{8\log n}{(r_{h,i}-\nu_1\rho^h)^2}+1+\sum_{t=z+1}^n\Big(\frac{t}{n^4}+\frac{\hat{H}}{t^3}\Big)\\
    &\leq\frac{8\log n}{(r_{h,i}-\nu_1\rho^h)^2}+1+\frac{1}{n^2}+\hat{H}\sum^n_{t=1}\frac{1}{t^3}\\
    &\leq\frac{8\log n}{(r_{h,i}-\nu_1\rho^h)^2}+2+2\hat{H},
\end{align*}
where the last inequality follows from the fact that $\sum_{t=1}^n\frac{1}{t^3}<2$.
\end{proof}

Having proved Lemma \ref{lemma:expected_t_hi_2}, it is now possible to finalise the proof of the main result. 

\subsection{Proof of the theorem}

The proof proceeds along the same four steps outlined in the manuscript. We specialize the bound in Lemma~\ref{lemma:expected_t_hi_2} to bound the number of times that suboptimal nodes in the tree are selected. The second step breaks down the value of $\EE{R_n}$ into three components. The third step bounds each component of the regret individually. The fourth step wraps the proof by combining the intermediate results into a final regret bound.

\vspace{2ex}

We revisit some notation. For a node $(h,i)$ in $\T_n$, let $f_{h,i}^*=\sup_{x\in\X_{h,i}}f(x)$, and define the set 
\begin{equation*}
\I_h=\set{(h,i)\mid f^*_{h,i}\geq f(x^*)-2\nu_1\rho^h},
\end{equation*}
i.e., $\I_h$ is the set of $2\nu_1\rho^h$-optimal nodes at depth $h$. Let $\J_h$ be the set of nodes at depth $h$ that are not in $\I_h$ but whose parents are in $\I_{h-1}$.

\paragraph*{First step.} We start by noting that we can use Lemma~\ref{lemma:expected_t_hi_2} to bound the expected number of times that a node $(h,i)\in\J_h$ is played. Noting that, for these nodes, $r_{h,i}>2\nu_1\rho^h$, then
\begin{equation}\label{eq:theorem_bound_Jh}
    \EE{T_{h,i}(n)}\leq\frac{8\log n}{(2\nu_1\rho^h-\nu_1\rho^h)^2}+2(\hat{H}+1)=\frac{8\log n}{\nu_1^2\rho^{2h}}+2(\hat{H}+1).
\end{equation}

We can also bound the cardinality of sets $\I_h$. Define
\begin{equation*}
\U_h=\cup_{c\in\I_h}\X_c.
\end{equation*}
In other words, $\U_h$ is the union of all cells $\X_{h,i}$ such that $(h,i)\in\I_h$. Since, by definition, for $(h,i)\in\I_h$, $r_{h,i}\leq2\nu_1\rho^h$, it follows from Lemma~\ref{lemma:subopt_factor} (with $c=2)$ that $\X_{h,i}\subset\X_{4\nu_1\rho^h}$ and, therefore, $\U_h\subset \X_{4\nu_1\rho^h}$. But then, the cardinality of $\I_h$ can be bounded as
\begin{align*}
\abs{\I_h} 
    \leq\N(\U_h,\ell,\nu_2\rho^h) 
    \leq\N(\X_{4\nu_1\rho^h},\ell,\nu_2\rho^h) 
    =\N(\X_{\frac{4\nu_1}{\nu_2}\nu_2\rho^h},\ell,\nu_2\rho^h).
\end{align*}

If $d$ is the $(4\nu_1/\nu_2,\nu_2)$-near-optimality dimension, by definition there is $C>0$ such that 
\begin{equation}\label{eq:ih_card_bound}
\abs{\I_h}\leq\N(\X_{\frac{4\nu_1}{\nu_2}\nu_2\rho^h},\ell,\nu_2\rho^h)\leq C(\nu_2\rho^h)^{-d}
\end{equation}

\paragraph*{Second step.} We partition the nodes of $\T_n$ into three disjoint sets, $\T^1$, $\T^2$, and $\T^3$, such that $\T^1\cup\T^2\cup\T^3=\T_n$. In particular, 
\begin{itemize}
\item $\T^1$ is the set of nodes at depth $\hat{H}$ that are $2\nu_1\rho^{\hat{H}}$-optimal, i.e., $\T^1=\I_{\hat{H}}$.
\item $\T^2$ is the sets of all $2\nu_1\rho^h$-optimal nodes, for all depths $h\neq\hat{H}$, i.e., $\T^2=\cup_{h=0}^{\hat{H}-1}\I_h$.
\item $\T^3$ is the set set of all remaining nodes, i.e., $\cup_{h=1}^{\hat{H}}\J_h$ and corresponding descendant nodes.
\end{itemize}
We now define the regret $R_{n,k}$ of playing actions in $\T^k, k=1,2,3$, as 
\begin{equation*}
R_{n,k}=\sum_{t=1}^n(f(x^*)-f(X_t))\II{(H_t,I_t)\in\T^k}
\end{equation*}
where $f(X_t)$ is the expected reward collected at time step $t$ given the action $X_t$ and $(H_t,I_t)$ is the (random) node selected in that same time step. It follows that
\begin{equation}\label{Eq:Regret-decomposed-1}
\EE{R_n}=\EE{R_{n,1}}+\EE{R_{n,2}}+\EE{R_{n,3}}.
\end{equation}

\paragraph*{Third step.} We now bound each of the terms in \eqref{Eq:Regret-decomposed-1} to establish a bound for the total cumulative regret after $n$ rounds.

Starting with $R_{n,1}$, in the worst case we have that
\begin{equation}\label{Eq:Bound-R1-1}
\EE{R_{n,1}}\leq4n\nu_1\rho^{\hat{H}}.
\end{equation}

In turn, the nodes $(h,i)\in\T^2$ are $2\nu_1\rho^h$-optimal. This means that
\begin{equation*}
\EE{R_{n,2}}\leq\sum_{t=1}^n4\nu_1\rho^{H_t}\II{(H_t,I_t)\in\T^2}
\end{equation*}
However, each node $(h,i)\in\T^2$ is selected only once and then expanded---only nodes at depth $\hat{H}$ can be selected multiple times. As such, we can use the bound in \eqref{eq:ih_card_bound} to get
\begin{equation}\label{Eq:Bound-R2-1}
\EE{R_{n,2}}
  \leq\sum_{h=0}^{\hat{H}-1}4\nu_1\rho^h\abs{\I_h}
  \leq4C\nu_1\nu_2^{-d}\sum_{h=0}^{\hat{H}-1}\rho^{h(1-d)}
\end{equation}

Finally, let us consider the nodes in $\T^3$. Each node $(h,i)\in\J_h$ is the child of a node in $\I_{h-1}$ and, as such, $f(x^*)-f^*_{h,i}\leq4\nu_1\rho^{h-1}$. Additionally, $\abs{\J_h}\leq 2\abs{\I_{h-1}}$, with equality if $\I_{h-1}$ was fully expanded and all the direct children of the nodes $(h-1,i)\in\I_h$ are in $\T_n$. 
 
Using \eqref{eq:theorem_bound_Jh}, we can now bound $\EE{R_{n,3}}$ as

\begin{align}
\nonumber%
\EE{R_{n,3}}
  &\leq\sum_{h=1}^{\hat{H}}4\nu_1\rho^{h-1}\sum_{(h,i)\in\J_h}\EE{T_{h,i}(n)}\\
\nonumber%
  &\leq\sum_{h=1}^{\hat{H}}4\nu_1\rho^{h-1}\abs{\J_h}\left(\frac{8\log n}{\nu_1^2\rho^{2h}}+2(\hat{H}+1)\right)\\
\nonumber%
  &\leq\sum_{h=1}^{\hat{H}}8\nu_1\rho^{h-1}\abs{\I_{h-1}}\left(\frac{8\log n}{\nu_1^2\rho^{2h}}+2(\hat{H}+1)\right)\\ 
\label{Eq:Bound-R3-1}%
  &\leq8C\nu_1\nu_2^{-d}\sum_{h=0}^{\hat{H}-1}\rho^{h(1-d)}\left(\frac{8\log n}{\nu_1^2\rho^{2h}}+2(\hat{H}+1)\right).
\end{align}

\paragraph*{Fourth step.} We now combine the bounds~\eqref{Eq:Bound-R1-1}, \eqref{Eq:Bound-R2-1} and \eqref{Eq:Bound-R3-1} to get
\begin{equation}\label{Eq:Combined-bound-1}
\EE{R_n}
  \leq4n\nu_1\rho^{\hat{H}}+4C\nu_1\nu_2^{-d}\sum_{h=0}^{\hat{H}-1}\rho^{h(1-d)}\left(\frac{8\log n}{\nu_1^2\rho^{2h}}+4\hat{H}+5\right)
\end{equation}
Rewriting \eqref{Eq:Combined-bound-1} using asymptotic notation, we finally get
\begin{equation*}
\EE{R_n}\in O\left(n\rho^{\hat{H}}+\log n\rho^{-\hat{H}(1+d)}+\hat{H}\max\{\rho^{\hat{H}(1-d)},\hat{H}\}\right),
\end{equation*}
and the proof is complete.
\end{document}